\documentclass{article}
\usepackage{amsfonts,amsmath,amssymb,amsthm,enumitem}

\PassOptionsToPackage{obeyspaces}{url}
\usepackage[colorlinks=true,citecolor=blue,urlcolor=blue,linkcolor=blue,bookmarksopen=true]{hyperref}
\usepackage{breakurl}
\usepackage{tikz}
\usetikzlibrary{calc}

\usepackage{etoolbox} 
\patchcmd{\thebibliography}{\leftmargin\labelwidth}{\leftmargin\labelwidth\addtolength\itemsep{-0.1\baselineskip}}{}{}

\author{Boris Bukh\thanks{Department of Mathematical Sciences, Carnegie Mellon University, Pittsburgh, PA 15213, USA. Supported in part by a Sloan Research Fellowship and by U.S.\ taxpayers through NSF CAREER grant DMS-1555149.} \and Michael Tait\thanks{Department of Mathematics \& Statistics, Villanova University, Villanova, PA, 19085, USA. This work was completed while the second author was a postdoc at Carnegie Mellon University and was supported in part by NSF grant DMS-1606350.}}

\title{Tur\'an numbers of theta graphs}
\date{}

\newtheorem{theorem}{Theorem}
\newtheorem{lemma}[theorem]{Lemma}

\newtheorem{definition}[theorem]{Definition}
\newtheorem{proposition}[theorem]{Proposition}

\newcommand*{\eqdef}{\stackrel{\text{\tiny{def}}}{=}}            
\newcommand*{\abs}[1]{\lvert #1\rvert}                           

\newcommand*{\dgr}{\mathop{\overrightarrow{\deg}}}               
\newcommand*{\dgl}{\mathop{\overleftarrow{\deg}}}                
\newcommand*{\Nr}{\overrightarrow{N}}                            
\newcommand*{\Nl}{\overleftarrow{N}}                             
\newcommand*{\Lr}{L}                                             
\newcommand*{\E}{\mathbb{E}}                                     
\newcommand*{\F}{\mathbb{F}}                                     
\DeclareMathOperator{\ex}{ex}

\newcommand*{\TODO}[1]{\marginpar{\usefont{T1}{fxb}{o}{b}\fontsize{24pt}{26pt}\selectfont!}}

\begin{document}
\maketitle

\begin{abstract}
The theta graph $\Theta_{\ell,t}$ consists of two vertices joined by $t$ vertex-disjoint paths of length $\ell$ each.
For fixed odd $\ell$ and large $t$, we show that the largest graph not containing $\Theta_{\ell,t}$ has at most $c_{\ell} t^{1-1/\ell}n^{1+1/\ell}$
edges and that this is tight apart from the value of $c_{\ell}$.
\end{abstract}

\section{Introduction}

Given a graph $F$, the {\em Tur\'an number for $F$}, denoted by $\ex(n, F)$ is the maximum number of edges in an $n$-vertex graph that contains no subgraph isomorphic to $F$. Mantel and Tur\'an determined this function exactly when $F$ is a complete graph, and the study of Tur\'an numbers has become a fundamental problem in combinatorics (see \cite{FurediSimonovits, Keevash, Sidorenko} for surveys). The Erd\H{o}s--Stone theorem \cite{ErdosStone} determines the asymptotic behavior of $\ex(n, F)$ whenever $\chi(F) \geq 3$, and so the most interesting Tur\'an-type problems are when the forbidden graph is bipartite. 

One of the most well-studied bipartite Tur\'an problems is the even cycle problem: the study of $\ex(n, C_{2\ell})$. Erd\H{o}s initiated the study of this problem when he needed an upper bound on $\ex(n, C_4)$ in order to prove a theorem in combinatorial number theory \cite{Erdos1938}. The combination of the upper bounds by K\H{o}vari, S\'os and Tur\'an \cite{KovariSosTuran} and the lower bounds by Brown \cite{Brown} and Erd\H{o}s, R\'enyi and S\'os \cite{ErdosRenyiSos} gave the asymptotic formula \[\ex(n, C_4) \sim \frac{1}{2}n^{3/2}.\]

It is now known that for certain values of $n$ the extremal graphs must come from projective planes \cite{Furedi1983, Furedi1996, WillifordREU} and this is conjectured to be the case for all $n$ (see \cite{Furedi1994}).

A general upper bound for $\ex(n, C_{2\ell})$ of $c_\ell n^{1+1/\ell}$ for sufficiently large $n$ was originally claimed by Erd\H{o}s \cite{Erdos} and first published by Bondy and Simonovits \cite{BondySimonovits} who showed that one can take $c_\ell = 20\ell$. Subsequent improvements of the best constant $c_\ell$ to $8(\ell-1)$ by Verstra\"ete \cite{Verstraete}, to $(\ell-1)$ by Pikhurko \cite{Pikhurko}, and to $80\sqrt{\ell}\log \ell$ by Bukh and Jiang \cite{BukhJiang} were made, and this final bound is the current record.

As stated above, we have an asymptotic formula for $\ex(n, C_4)$. Additionally, the upper bound on $\ex(n,C_{2\ell})$ is of the correct order of magnitude for $\ell\in \{3,5\}$ \cite{Benson,Wenger}, i.e., $\ex(n,C_{2\ell})=\Theta(n^{1+1/\ell})$ for these values of~$\ell$. However, unlike the case of $C_4$, the sharp multiplicative constant is not known; see \cite{FurediNaorVerstraete} for the best bounds on $\ex(n, C_6)$. The order of magnitude for $\ex(n, C_{2\ell})$ is unknown for any $\ell \notin \{2,3,5\}$. 
The best known general lower bounds are given by Lazebnik, Ustimenko and Woldar \cite{LazebnikUstimenkoWoldar} (but see \cite{TerlepWilliford} for a better bound for the $\ex(n,C_{14})$ case).

Although it is unclear whether $\ex(n, C_{2\ell}) = \Omega(n^{1+1/\ell})$ holds in general, more is known if instead of forbidding a pair of internally disjoint paths of length $\ell$ between pairs of vertices (that is, a $C_{2\ell}$) one forbids several paths of length~$\ell$ between pairs of vertices. For $t\in \mathbb{N}$, let $\Theta_{\ell, t}$ be the graph made of~$t$ internally disjoint paths of length~$\ell$ connecting two endpoints. The study of $\ex(n, \Theta_{\ell, t})$ generalizes the even cycle problem as $\Theta_{\ell, 2} = C_{2\ell}$. Faudree and Simonovits showed \cite{FaudreeSimonovits} that 
\[
\ex(n, \Theta_{\ell, t}) = O_{\ell, t} \left(n^{1+1/\ell}\right).
\]

More recently, Conlon showed that this upper bound gives the correct order of magnitude if the number of paths is a large enough constant \cite{Conlon}. That is, there exists a constant $c_\ell$ such that $\ex(n, \Theta_{\ell, c_{\ell}}) = \Theta_\ell(n^{1+1/\ell})$. Verstra\"ete and Williford \cite{VerstraeteWilliford} constructed graphs with no $\Theta_{4,3}$ that have $(\frac{1}{2} - o(1))n^{5/4}$ edges.

In this paper, we are interested in the behavior of $\ex(n, \Theta_{\ell, t})$ when $\ell$ is fixed and $t$ is large. When $\ell=2$, a result of F\"uredi \cite{FurediBipartite} shows that $\ex(n,\Theta_{2,t}) \sim \tfrac{1}{2}\sqrt{t} n^{3/2}$. For general~$\ell$, the result of Faudree and Simonovits gives that $\ex(n, \Theta_{\ell, t}) \leq c_\ell t^{\ell^2}n^{1+1/\ell}$. We improve this bound as follows.

\begin{theorem}\label{thm:main}
For fixed $\ell\geq 2$, we have 
\[
  \ex(n,\Theta_{\ell,t})=O_{\ell}\left(t^{1-1/\ell} n^{1+1/\ell}\right).
\]
\end{theorem}

When~$\ell$ is odd, we show that the dependence on~$t$ in Theorem~\ref{thm:main} is correct.
\begin{theorem}\label{thm:oddconstruction}
Let $\ell\geq 3$ be a fixed odd integer. Then
\[
\ex(n, \Theta_{\ell, t}) = \Omega_\ell\left(t^{1-1/\ell}n^{1+1/\ell}\right).
\]
\end{theorem}

We do not know if Theorem~\ref{thm:main} is tight when $\ell$ is even. In this case, our best lower bound is the following.
\begin{theorem}\label{thm:evenconstruction}
Let $\ell\geq 2$ be a fixed even integer. Then
\[
 \ex(n, \Theta_{\ell, t}) = \Omega_\ell(t^{1/\ell} n^{1+1/\ell}).
\]
\end{theorem}
It would be interesting to close the gap between Theorems~\ref{thm:main} and~\ref{thm:evenconstruction} for even~$\ell$.

Since the proof of Theorem~\ref{thm:main} is relatively involved, we begin by introducing the main ideas in Section~\ref{sec:three} where we prove
the theorem in the case~$\ell=3$. Then in Sections \ref{sec:general} and \ref{sec:embedding} we extend this argument to prove the general upper bound. 
In Sections \ref{sec:oddconstruction} and~\ref{sec:evenconstruction} we give constructions for odd and even values of $\ell$ respectively.\medskip

\textbf{Related work.} A year after this work was completed Xu, Zhang, and Ge \cite{XuZhangGe} showed that a variant of the construction
in Theorem~\ref{thm:evenconstruction} demonstrates sharpness of the constant in the K\"ovari--S\'os--Tur\'an theorem.

\section{Case \texorpdfstring{$\ell=3$}{l=3}}\label{sec:three}
In this section we present the proof of Theorem~\ref{thm:main}, dealing 
with the case $\ell=3$. As every graph of average degree $4d$ contains a bipartite subgraph of average degree $2d$, and since
every graph of average degree $2d$ contains a subgraph of minimum degree $d$, we henceforth assume that the graph is bipartite of minimum degree~$d$. 

\begin{lemma}\label{lem:nottoobad}
Let $r$ be any vertex of $G$. Call a vertex $u$ \emph{bad} if $u\neq r$ and $u$ has more than~$t$ common neighbors
with~$r$. If $G$ is $\Theta_{3,t}$-free, then no neighbor of~$r$ is adjacent to~$t$ bad vertices.
\end{lemma}
\begin{proof}
Suppose $w$ is adjacent to bad vertices $u_1,\dotsc,u_t$. Define a sequence of vertices $z_1,\dotsc,z_t$ as follows.
We let $z_i$ be any common neighbor of $r$ and $u_i$ other than $w,z_1,\dotsc,z_{i-1}$. It exists since there are more than~$t$
common neighbors between $r$ and $u_i$. Then $(wu_iz_ir)_{i=1}^t$ is a collection of $t$ disjoint paths of length $3$ from
$w$ to~$r$.
\end{proof}

\begin{proof}[Proof of Theorem~\ref{thm:main} for $\ell=3$]
Let $r$ be any vertex of $G$. Let $\Lr_0=\{r\}$. Let $\Lr_1$ be the set of all the neighbors of $r$.
Let $\Lr_2$ be the set of all vertices at distance $2$ from $r$ that have at most $t$ common neighbors
with $r$. Note that by Lemma~\ref{lem:nottoobad} each vertex in $\Lr_1$ has at least $d-t$ neighbors in~$\Lr_2$.
Call a vertex $v_1\in \Lr_1$ a \emph{parent} of $v_2\in \Lr_2$ if $v_1$ and $v_2$ are adjacent. Note that
a vertex in $L_2$ can have at most $t$ parents. Hence, each vertex in $\Lr_2$ has at least $d-t$
neighbors in $V(G)\setminus \Lr_1$.

Let $\Lr_3$ be all vertices in $V(G)\setminus \Lr_1$ that are adjacent to some
$\Lr_2$. Call $v_3\in \Lr_3$ a \emph{descendant} of $v_1\in \Lr_1$ if there is
a path of the form $v_1v_2v_3$ with $v_2\in \Lr_2$.

Let $B(v_1)\subset \Lr_3$ be the set of all the descendants of $v_1$
that have more than $t$ common neighbors with $v_1$. By Lemma~\ref{lem:nottoobad},
each $v_2\in N(v_1)$ has fewer than~$t$ neighbors in $B(v_1)$. 

Let $H$ be the subgraph of $G$ obtained from $G$ by removing all edges
between $B(v_1)$ and $N(v_1)$ for all $v_1\in \Lr_1$. For a vertex $v_2\in \Lr_2$, an edge incident to it is removed only when $v_2$ is adjacent to some $v_1\in \Lr_1$ and the other endpoint of this edge is a neighbor of $v_2$ in $B(v_1)$. We noted above that by Lemma~\ref{lem:nottoobad},
each $v_2\in N(v_1)$ has fewer than~$t$ neighbors in $B(v_1)$. Therefore since each $v_2\in \Lr_2$
has at most $t$ parents, each vertex in $\Lr_2$ has at least $d-t-t(t-1)=d-t^2$ neighbors
in~$\Lr_3$.

For a vertex $v_3\in \Lr_3$, let $p(v_3)$ be the number of paths (in $H$) of the form
$rv_1v_2v_3$ with $v_i\in \Lr_i$. We claim that $p(v_3)\leq 2t(t-1)$ for every
$v_3\in \Lr_3$. Indeed, suppose the contrary. We will construct a $\Theta_{3,t}$ subgraph as follows.
First, we pick any path $rv_1^{(1)}v_2^{(1)}v_3$ counted by~$p(v_3)$.
Since $v_3$ and $v_1^{(1)}$ have at most $t$ common neighbors, and since $r$ and $v_2^{(1)}$
have at most $t$ common neighbors, at most $2t$ paths counted by $p(v_3)$ intersect
$\{v_1^{(1)},v_2^{(2)}\}$.
So, we can pick another path $rv_1^{(2)}v_2^{(2)}v_3$ that is disjoint from $\{v_1^{(1)},v_2^{(2)}\}$.
We can repeat this, at each step selecting path $rv_1^{(i)}v_2^{(i)}v_3$ that is disjoint from $\bigcup_{j<i}\{v_1^{(j)},v_2^{(j)}\}$
for $i=1,\dotsc,t$. The paths $rv_1^{(i)}v_2^{(i)}v_3$ together form a $\Theta_{3,t}$. So, $p(v_3)\leq 2t(t-1)$ after all.

Since each vertex in $\Lr_1$ has at least $d-t$ neighbors in $\Lr_2$ and each vertex in $\Lr_2$ has at least $d-t^2$ neighbors in $L_3$,
it follows that
\[
  \abs{\Lr_3}\geq \frac{d(d-t)(d-t^2)}{2t(t-1)}.
\]
Since $|\Lr_3| \leq n$ the result follows.
\end{proof}

\section{General case}\label{sec:general}
\paragraph{Outline.}
The case of general $\ell$ is similar to the case $\ell=3$. Starting with a root
vertex, we build a sequence of layers $\Lr_1,\Lr_2,\dotsc,\Lr_{\ell}$ such that
each next layer is about~$d$ times larger than the preceding. The condition
of being $\Theta_{\ell,t}$-free is used to ensure that a vertex in~$\Lr_j$
descends from a vertex in~$\Lr_i$ in at most $O(t^{j-i-1})$ ways.
However, there are two complications that are not present in
the proof of the $\ell=3$ case.

First, in the definition of $\Lr_2$ we excluded
vertices that have too many neighbors back. Doing so affects degrees of yet-unexplored vertices, 
such as those in $\Lr_3$. That was not important for the $\ell=3$ case because $\Lr_3$ was the final layer. In general, though, we will maintain a set of `bad' vertices and will control 
how removal of these vertices affects subsequent layers.

Second, removing vertices from later layers reduces degrees of the vertices in the preceding
layers. So, instead of trying to ensure that each vertex has large degree, we will maintain
a weaker condition that there are many paths from the root to the leaves of the tree.

\paragraph{Minimum and maximum degree control.} As in the proof of the case $\ell=3$, we will need to ensure that
all vertices are of large degree. For technical reasons, which will become apparent in Section~\ref{sec:embedding}
we need to control not only the minimum, but also the maximum degrees. This is done with the help of the following
lemma:
\begin{lemma}[Theorem~12 in \cite{BukhJiang}, only in arXiv version]
  Every $n$-vertex graph with $\geq 6\ell c n^{1+1/\ell}$ edges contains a subgraph $G$ such that
\begin{itemize}
\item The graph $G'$ has at least $cn^{1/2\ell}$ vertices, and
\item Degree of each vertex of $G'$ is between $c v(G')^{1/\ell}$ and $\Delta c v(G')^{1/\ell}$ where $\Delta=(20\ell)^{2\ell}$.
\end{itemize}
\end{lemma}

Henceforth we assume that our graph is bipartite, and that each vertex has degree between $d$ and~$\Delta d$, where $\Delta$ is as above.
We will show that $d^{\ell}\leq (8\ell t)^{\ell-1} n\bigl(1+o_{\ell}(1)\bigr)$, and hence that every $\Theta_{\ell,t}$-free graph
has at most $96\ell^2 t^{1-1/\ell}n^{1/\ell}\bigl(1+o_{\ell}(1)\bigr)$ edges (the factor of $6\ell$ is from the preceding lemma,
and another factor of $2$ is because of passing to a bipartite subgraph). 

\paragraph{Graph exploration process.} We shall use the same terminology as in the case $\ell=3$.
Namely, if $v\in \Lr_i$ and $u\in \Lr_{i+1}$ are neighbors, then we say that $v$ is a \emph{parent} of $u$
and $u$ is a \emph{child} of $v$. A path of the form $v_iv_{i+1}\dotsb v_j$ with $v_s\in \Lr_s$ for $i\leq s \leq j$
is called a \emph{linear path}; vertex $v_j$ is called a \emph{descendant} of~$v_i$.
We let $P(v_i,v_j)$ denote the number of linear paths from $v_i$ to~$v_j$. For sets $A\subset \Lr_i$
and $B\subset \Lr_j$, we denote by $P(A,B)$ the number of linear paths going from a vertex in $A$ 
to a vertex in~$B$.

In addition to sets $\Lr_0,\Lr_1,\dotsc,\Lr_k$ we will also maintain a sets~$B_1,\dotsc,B_{k-1}$ of \emph{bad} vertices.
All the sets $\Lr_0,\Lr_1,\dotsc,\Lr_k, B_1,\dotsc,B_{k-1}$ will be disjoint. 
We shall say that we are at \emph{stage}~$k$
if the sets $\Lr_0,\Lr_1,\dotsc,\Lr_k$ and $B_1,\dotsc,B_{k-1}$ have been defined, but the sets $\Lr_{k+1}$ and $B_k$ have not yet been defined.
We denote by $U\eqdef V(G)\setminus(L_0\cup \dotsb\cup L_k\cup B_1\cup\dotsb \cup B_{k-1})$
the set of \emph{unexplored} vertices.

For $v\in \Lr_i$, let $\Nr(v)$ be the set of children of $v$, and let $\Nl(v)$ be the set of parents of $v$.
We define $\dgr(v)\eqdef \abs{\Nr(v)}$ and $\dgl(v)\eqdef \abs{\Nl(v)}$ 
to be the number of children and parents of $v$, respectively.
The reason for this notation is that we imagine $L_0,\dotsc,L_k$ grow from left to right.
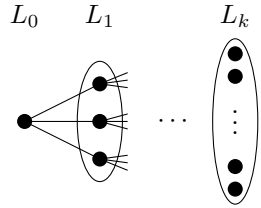
\begin{center}
\begin{tikzpicture}
\def\xspacing{1}      
\def\lastxspacing{2.8}  
\def\ylabel{1.4}      
\def\vertexsize{0.1}  
\fill (0,0) circle (\vertexsize);
\draw (\xspacing,0) ellipse (0.3 and 0.8); 
\foreach \y in {-0.5,0,0.5}
{
  \fill (\xspacing,\y) circle (\vertexsize);
  \draw (0,0)--(\xspacing,\y);
  \foreach \dy in {0.1,0,-0.1}
  \draw (\xspacing,\y) -- +($(0.37,\dy)+0.1*(0,\y)$);
}
\node at ($0.5*(\xspacing,0)+0.5*(\lastxspacing,0)+(0.1,0)$) {$\cdots$};
\draw (\lastxspacing,0) ellipse (0.3 and 1.1);
\node at (\lastxspacing,0) {\raisebox{1.25ex}{$\vdots$}}; 
\foreach \y in {-0.9,-0.6,0.6,0.9}
{
  \fill (\lastxspacing,\y) circle (\vertexsize);
}
\node at (0.0, \ylabel) {$\Lr_0$}; 
\node at (\xspacing, \ylabel) {$\Lr_1$}; 
\node at (\lastxspacing, \ylabel) {$\Lr_k$}; 
\end{tikzpicture}
\end{center}

Let
\[
  R_m\eqdef \frac{(2\ell)^m}{m+1}\binom{2m}{m}t^m.
\]
Note that $\frac{1}{m+1}\binom{2m}{m}$ is the $m$'th Catalan number.
We call a pair of layers $(\Lr_i,\Lr_j)$ with $i<j$ \emph{regular} if for every pair
of vertices $(v_i,v_j)\in \Lr_i\times \Lr_j$ the number of linear paths from~$v_i$ to~$v_j$
is $P(v_i,v_j)\leq R_{j-i-1}$.

We start the exploration process by picking a root vertex $r$, and setting $\Lr_0=\{r\}$ and $\Lr_1=N(r)$.
At the $k$th stage the sets $B_1,B_2,\dotsc,B_{k-1},\Lr_0,\Lr_1,\dotsc,\Lr_k$ satisfy the following properties:
\begin{enumerate}[label=P\arabic*.,ref=P\arabic*]
\item \label{prop:root} The root is preserved: $\Lr_0=\{r\}$.
\item \label{prop:orphans} No orphans: every vertex of $\Lr_i$ for $i=1,2,\dotsc,k$ has at least one parent.
\item \label{prop:tree} The explored part is tree-like: every pair of layers $(\Lr_i,\Lr_j)$
with $0\leq i<j<k$ is regular. (Note that pairs of the form $(\Lr_i,\Lr_k)$ might be irregular.) 
\item \label{prop:bad} Bad sets are small: $\abs{B_j}\leq \tau_j d^{j-1}$ for all $1\leq j< k$, where $\tau_j\eqdef 2\ell t\sum_{i=0}^{j-1}(i+1) \Delta^i$. 
\item \label{prop:branching} The `tree' is growing: there are at least $d^{k-1}(d-\eta_k)$ linear paths from root $r$
to layer~$\Lr_k$, where $\eta_k\eqdef\nobreak \sum_{i=0}^{k-2} \bigl((\Delta+1)R_i+2(i+1)\ell t\Delta ^i+\tau_i\bigr)$.
\item \label{prop:unexplored} There are many children: each vertex in $\Lr_k$ has at least $d$ neighbors in $\Lr_{k-1}\cup B_{k-1}\cup U$,
and each vertex in $U$ has at least $d$ neighbors in $\Lr_k\cup U$.
\end{enumerate}

The main step in going from the $k$th stage to the $(k+1)$st is to make \ref{prop:tree}
hold with $j=k$. 
To that end, we rely on the following lemma showing that only a few $v_k\in \Lr_k$ are in pairs $(v_i,v_k)$
that violate \ref{prop:tree}.

\begin{lemma}[Proof is in Section~\ref{sec:embedding}]\label{lem:fewexceptions}
Let $B'\eqdef \{ v_k\in\Lr_k : \exists i<k\ \exists v_i\in \Lr_i\ P(v_i,v_k)>R_{k-i-1}\}$.
Then $P(r,B')\leq 2k \ell t(\Delta d)^{k-1}$.
\end{lemma}
Assuming the lemma, we show next how to go from stage $k$ to stage $k+1$, for $k<\ell$.

Because of \ref{prop:orphans}, $P(r,v)\geq \dgl(v)$ and so $\dgl(v_k)\leq R_{k-1}$ 
for every $v_k\in \Lr_k\setminus B'$. Since degree of every vertex of $\Lr_k$ is at least $d$, by \ref{prop:unexplored}
this implies that every vertex in $\Lr_k\setminus B'$ has at least $d-R_{k-1}$ neighbors in $U\cup B_{k-1}$.
Let $B''$ consist of those vertices in $\Lr_k\setminus B'$ that have at least $\Delta R_{k-1}$
neighbors in~$B_{k-1}$. Note that
\[
  \Delta R_{k-1} \abs{B''}\leq d\Delta \abs{B_{k-1}},
\]
and hence $\abs{B''}\leq d \abs{B_{k-1}}\leq \frac{\tau_{k-1}}{R_{k-1}}d^{k-1}$.

Let $\Lr_{k+1}$ be all vertices in $U$ that are adjacent to some vertex in $\Lr_k\setminus (B'\cup B'')$,
replace $\Lr_k$ by $\Lr_k\setminus (B'\cup B'')$, and set $B_k\eqdef B'\cup B''$. That way, each linear path
from $r$ to $\Lr_k$ can be extended to a path to $\Lr_{k+1}$ in at least $d-R_{k-1}-\Delta R_{k-1}=d-(\Delta+1)R_{k-1}$ ways. 
So, since the number of linear paths from $r$ to the new $\Lr_k$ is at least
\begin{align*}
  d^{k-1}(d-\eta_k)-P(r,B')-P(r,B'')&\geq d^{k-1}(d-\eta_k-2k\ell t \Delta^{k-1})-R_{k-1}\abs{B''}\\
                                   &\geq d^{k-1}(d-\eta_k-2k\ell t \Delta^{k-1}-\tau_{k-1}),
\end{align*}
it follows that the number of linear paths from $r$ to $\Lr_{k+1}$ is at least
\begin{align*}
  \bigl(d-(\Delta+1)R_{k-1}\bigr)P(r,\Lr_k)&\geq d^k\bigl(d-\eta_k-2k\ell t \Delta^{k-1}-\tau_{k-1}-(\Delta+1)R_{k-1}\bigr)\\
                                          &=d^k(d-\eta_{k+1}) 
\end{align*}
This shows that \ref{prop:branching} holds at stage~$k+1$. 
Since \ref{prop:orphans} held at stage $k$, it follows that $\abs{B'}\leq P(r,B')$, implying
\[
  \abs{B_k}=\abs{B'}+\abs{B''}\leq 2k \ell t (\Delta d)^{k-1}+\tau_{k-1}d^{k-1}=\tau_k d^{k-1},
\]
and so Property~\ref{prop:bad} holds at stage~$k+1$.
Property \ref{prop:unexplored} holds at stage $k+1$ because it held at stage $k$ and the graph is bipartite. The other properties are immediate.

At $\ell$'th stage, the number of linear paths from $r$ to $\Lr_{\ell}$ is at least $d^{\ell-1}(d-\eta_{\ell})=d^{\ell}\bigl(1+o(1)\bigr)$.
On the other hand, it is at most $\abs{\Lr_{\ell}}R_{\ell-1}\leq (8\ell t)^{\ell-1}n$.  The result then follows.

\section{Embedding \texorpdfstring{$\Theta_{\ell,t}$}{Theta\textunderscore l,t}}\label{sec:embedding}
In this section we prove Lemma~\ref{lem:fewexceptions} that controls the number of linear
paths in a $\Theta_{\ell,t}$-free graph. For that we show that if there are many
linear paths from some vertex $v$ to its descendants, then we can embed
a subdivision of a star so that its leaves are mapped to the children of $v$.
Adding vertex $v$ to the subdivision of the star yields a copy of $\Theta_{\ell,t}$.

The standard method for embedding trees is to find a substructure of large `minimum degree' (in a suitable sense),
and then embed vertex-by-vertex avoiding already-embedded vertices. For us the relevant
notion of a degree is the number of linear paths. 

\begin{definition}
A pair of layers $(\Lr_i,\Lr_j)$ with $i<j$ is \emph{almost-regular} if every pair of layers $(L_{i'},L_{j'})$,
with $i\leq i'<j'\leq j$ and $(i',j')\neq (i,j)$, is regular. 
\end{definition}
\begin{lemma}\label{lem:embedding}
Suppose $i<k$ and pair $(\Lr_{i-1},\Lr_k)$ is almost-regular, and we are given a vertex $v_{i-1}\in \Lr_{i-1}$ and subsets $A\subset \Nr(v_{i-1})$,
$B\subset \Lr_k$. 
Suppose that the number of linear paths between $A$ and $B$ satisfies
\begin{align}
P(A,B)/\abs{A}&>2\ell t (\Delta d)^{k-i-1},\label{eq:Alower}\\
P(A,B)/\abs{B}&> R_{k-i}.\label{eq:Blower}
\end{align}
Then~$G$ contains $\Theta_{\ell,r}$.
\end{lemma}
\begin{proof} 
The proof naturally breaks into three parts: finding a substructure of a large minimum degree,
using that substructure to locate many disjoint paths, and then joining these paths to form a copy of $\Theta_{\ell,t}$.

\paragraph{Part 1 (large minimum degree substructure):}
We will select a subset $B'\subset B$ that is well connected by linear paths to the preceding layer~$\Lr_{k-1}$.
We use a modification of the standard proof that a graph of average degree~$2d$ contains a subgraph
of minimum degree~$d$.

At start, set $B'=B$. To each pair $(a,b)\in A\times B'$ we associate a set $\mathcal{P}(a,b)$
of linear paths between $a$ and~$b$. At start, $\mathcal{P}(a,b)$ is the set of all linear paths from $a$ to $b$.
For brevity we use notations $\mathcal{P}(\cdot,b)\eqdef \bigcup_{a\in A} \mathcal{P}(a,b)$, $\mathcal{P}(a,\cdot)\eqdef \bigcup_{b\in B'} \mathcal{P}(a,b)$,
and similarly for $\mathcal{P}(\cdot,\cdot)$.

Perform the following two operations, for as long as any of them is possible to perform:
\begin{enumerate}
\item \label{op:b} if $\abs{\mathcal{P}(\cdot,b)}\leq R_{k-i}/2$ for some $b\in B'$, then remove vertex $b$ from $B'$,
\item \label{op:c} if some linear path $a v_{i+1} v_{i+2}\dotsb v_{k-1}$ is a prefix of fewer than $\ell t$ paths in $\mathcal{P}(a,\cdot)$,
remove all these paths from respective $\mathcal{P}(a,b)$'s.
\end{enumerate}

Since each step decreases the size of $\mathcal{P}(\cdot,\cdot)$, the process terminates. 

Since operation~\ref{op:b} is performed at most $\abs{B}$ times, the operation decreases the size of $\mathcal{P}(\cdot,\cdot)$ by no more
than $\abs{B} R_{k-i}/2 < P(A,B)/2$.

Since each vertex has degree at most $\Delta d$, operation~\ref{op:c} is performed on at most
$(\Delta d)^{k-i-1} \abs{A}$ linear paths terminating in the layer~$\Lr_{k-1}$. Therefore, the operation decreases $\abs{\mathcal{P}(\cdot,\cdot)}$ by less
than 
\[
  (\Delta d)^{k-i-1} \abs{A} \cdot \ell t < P(A,B)/2.
\]

So, the total number of edges removed by the two operations is less than $P(A,B)$, and so $\mathcal{P}(\cdot,\cdot)$ is
non-empty when the process terminates. Therefore, $B'$ is non-empty as well. \medskip

\paragraph{Part 2 (many disjoint paths from  vertices of $B'$):}
Next we use the obtained set $B'$ and $\mathcal{P}(\cdot,\cdot)$ to embed~$\Theta_{\ell,t}$. 
We start by proving that, for every vertex $b\in B'$, there are $\ell t$ linear
paths in $\mathcal{P}(\cdot,b)$ that are vertex-disjoint apart from sharing vertex $b$ itself. We will
pick these paths one-by-one subject to the constraint of being vertex-disjoint.

Indeed, consider any linear path $v_iv_{i+1}\dotsc v_{k-1}b\in \mathcal{P}(\cdot,b)$.
Because $(\Lr_{i-1},\Lr_k)$ is almost-regular
the number of paths in $\mathcal{P}(\cdot,\cdot)$ that intersect $\{v_i,v_{i+1},\dotsc,v_{k-1}\}$ 
is at most 
\[
  \sum_{j=i}^{k-1} P(v_{i-1},v_j) P(v_j,v_k)\leq \sum_{j=i}^{k-1} R_{j-i}\cdot R_{k-j-1}=\sum_{u+v=k-i-1} R_u R_v=  \frac{1}{2\ell t} R_{k-i},
\]
where the last equality relies on the convolution identity for the Catalan numbers.

From $\abs{\mathcal{P}(\cdot,b)}>R_{k-i}/2$ it follows that as long we have picked fewer than $\ell t$ paths, there is
another path in $\mathcal{P}(\cdot,b)$ that is disjoint from the already-picked.

\paragraph{Part 3 (embedding):}
Let 
\[
  S\eqdef \{ v_{k-1}\in \Lr_{k-1} : v_{k-1} \text{ is on some path in }\mathcal{P}(\cdot,\cdot)\}.
\]

Consider the subgraph $H$ of $G$ that is induced by $S\cup B'$. This is a bipartite graph with parts $S$ and $B'$. 
The vertex-disjoint paths found in the previous step show that degree of each vertex in $B'$ is at least $\ell t$.
We claim that vertices of $S$ are also of degree at least~$\ell t$. Indeed, let $s\in S$ be arbitrary.
Then $\mathcal{P}(\cdot, \cdot)$ contains a linear path of the form $a v_{i+1} v_{i+2}\dotsb v_{k-2} s b$. Since it was not removed
by operation~\ref{op:c}, there are at least $\ell t$ linear paths having $a v_{i+1} v_{i+2}\dotsb v_{k-2} s$ as a prefix.
Therefore, $s$ is adjacent to at least $\ell t$ vertices of $B'$.

Because the minimum degree of~$H$ is at least~$\ell t$, it is possible to embed any rooted tree on at most $\ell t$ vertices into~$V(H)=S\cup B'$
with the root as any prescribed vertex of~$S\cup B'$.
In particular, we can find a vertex $u\in S\cup B'$ and $t$ vertex-disjoint paths from $u$ to $B'$ of length $\ell-k+i-1$ each.
Note that the choice of whether $u\in S$ or $u\in B'$ depends on the parity of $\ell-k+i-1$. 

Let $b_1,\dotsb,b_t\in B'$ be the endpoints of these paths, and let $T$ be all the vertices in the union of the paths. 
Since $\abs{T}<\ell t$, at least one of the $\ell t$ vertex-disjoint paths from $b_1$ to $A$ misses $T$. We
then join this path to~$b_1$. We can extend paths ending at $b_2,b_3,\dotsc,b_t$ in turn in a similar way.
We obtain an embedding of $\Theta_{\ell,t}$ minus one vertex. Adding $v_{i-1}$ we obtain an embedding
of $\Theta_{\ell,t}$ into~$G$.
\end{proof}

We are now ready to prove Lemma~\ref{lem:fewexceptions} that controls the number of bad vertices.
\begin{proof}[Proof of Lemma~\ref{lem:fewexceptions}]
Inductively define sets $B'_{k-1},B'_{k-2},\dotsc,B'_1$ (in that order) by
\[
  B'_i\eqdef \{ v_k\in \Lr_k : \exists v_{i-1}\in \Lr_{i-1}\ \text{s.t.\ } P(v_{i-1},v_k)>R_{k-i}\}\setminus (B'_{i+1}\cup\dotsb\cup B'_{k-1}).
\]
Note that $B'=\bigcup_i B'_i$. We will prove that $P(r,B'_i)\leq 2\ell t(\Delta d)^{k-1}$, from which the lemma would follow.

Decompose $B'_i$ further into sets 
\[
  B'(v_{i-1})\eqdef \{ v_k\in \Lr_k : P(v_{i-1},v_k)>R_{k-i}\}\setminus (B'_{i+1}\cup\dotsb\cup B'_{k-1}).
\]
Clearly $B'_i=\bigcup_{v_{i-1}\in\Lr_{i-1}} B'(v_{i-1})$.

To start, observe that if we remove $B'_{i+1}\cup\dotsb\cup B'_{k-1}$ from $\Lr_k$ then the pair of layers $(\Lr_{i-1},\Lr_k)$ is almost-regular.
Therefore, for every $v_{i-1}\in \Lr_{i-1}$, since
\begin{align*}
  P\bigl(\Nr(v_{i-1}),B'(v_{i-1})\bigr)&>R_{k-i}\abs{B'(v_{i-1})},
\intertext{it follows from Lemma~\ref{lem:embedding} that}
  P\bigl(\Nr(v_{i-1}),B'(v_{i-1})\bigr)&\leq 2\dgr(v_{i-1}) \ell t (\Delta d)^{k-i-1}.
\end{align*}
In particular,
\begin{align*}
  P(r,B'_i)&=\sum_{v_{i-1}\in \Lr_{i-1}} P(r,v_{i-1})P(\Nr(v_{i-1}),B'(v_{i-1}))\leq P(r,\Lr_{i-1}) 2\ell t(\Delta d)^{k-i}\\
           &\leq 2\ell t(\Delta d)^{k-1}
\end{align*}
since degree of every vertex is at most $d\Delta$. Adding these over all $i$ completes the proof.
\end{proof}

\section{Lower bound for odd length paths}\label{sec:oddconstruction}
In this section we will construct graphs on $n$ vertices that do not contain a $\Theta_{\ell, t}$ with $\Omega_\ell\left( t^{1-1/\ell}n^{1+1/\ell}\right)$ edges when $\ell$ is odd, showing that Theorem~\ref{thm:main} has the correct dependence on $t$ for odd~$\ell$. 

We will use the random polynomial method \cite{BlagojevicBukhKarasev,Bukh}. 
Our construction is in two stages. First we use random polynomials to construct graphs with only few short cycles.
In the second stage we blow up the graph by replacing vertices by large independent sets. We will show that the resulting
graph is $\Theta_{\ell,t}$-free.

Let $q$ be a prime power and let $\mathcal{P}_{d}^s$ be the set of polynomials in $s$ variables of degree at most $d$ over $\F_q$. That is, $\mathcal{P}_d^s$ is the set of linear combinations over $\F_q$ of 
monomials $X_1^{a_1}\cdots X_s^{a_s}$ with $\sum_{i=1}^s a_i \leq d$.

We reserve the term \emph{random polynomial} to mean a polynomial chosen uniformly from $\mathcal{P}_d^s$. We note that a random polynomial can equivalently be obtained by choosing the coefficients of each monomial $X_1^{a_1}\cdots X_s^{a_s}$ uniformly and independently from~$\F_q$. In particular, because the constant term of a random polynomial is chosen uniformly from $\F_q$, it follows that 
\begin{equation}\label{eq:edgeprobability}
\Pr[f(x) = 0] = \frac{1}{q}
\end{equation}
for a random polynomial $f$ and any fixed $x\in \F_q^s$.

We now define a random graph model that we use in our constructions.
\begin{definition}[Random algebraic graphs]\label{def:randomgraphs}
Set $d\eqdef 2\ell^2$. Let $U$ and $V$ be disjoint copies of $\F_q^{\ell}$, and consider the following random bipartite graph with parts $U$ and $V$.
We pick $\ell-1$ independent random polynomials $f_1,\dotsc,f_{\ell-1}$ from $\mathcal{P}_d^{2\ell}$, and declare $uv$ to be an edge of $G$ if and only if
\[
  f_1(u,v) = f_2(u,v) = \cdots = f_{\ell-1}(u,v) = 0.
\]
We call the resulting graph a \emph{random algebraic graph}.
\end{definition}
Note that in this definition we fixed the degree and number of polynomials, to suit our particular application. More general random algebraic graphs
have been been used for instance in \cite{BukhConlon} .

Let $G$ be a random algebraic graph. For $T\in \mathbb{N}$, we say that a pair of vertices $x,y$ is \emph{$T$-bad} if there are at least $T$ distinct (but not necessarily edge-disjoint) paths of length at most $\ell$ between $x$ and~$y$.
Define $\mathcal{B}_T$ to be the set of $T$-bad pairs of vertices in $G$. 

\begin{proposition}\label{prop:fewpaths}[Case $h=1$ of Proposition \ref{prop:fewpaths2}]
There exists a constant $T=T(\ell)$ depending only on $\ell$ such that 
\[
\E\bigl[\abs{\mathcal{B}_T}\bigr] = O_\ell(1).
\]
\end{proposition}
The proof of Proposition~\ref{prop:fewpaths2} is similar to arguments in \cite{BukhConlon} and \cite{Conlon}, and we defer it to Section~\ref{sec:boring}. We use Proposition~\ref{prop:fewpaths} to make a graph with $\Omega\left(n^{1+1/\ell}\right)$ edges where each pair of vertices is joined by only a few short paths.

\begin{theorem}\label{thm:fewpaths}
There exists a constant $T$ such that, for all $n$ large enough, there is a bipartite graph on $n$ vertices with at least $\frac{1}{4}n^{1+1/\ell}$ edges and no $T$-bad pair.
\end{theorem}

\begin{proof}
Let $q$ be the largest prime power with $2q^\ell \leq n$. Note that $2q^{\ell}\sim n$, as there is a prime between $x$ and $x+x^{0.525}$ for all large $x$ \cite{BakerHarmanPintz}. Let $G$ be a random algebraic graph
as in Definition~\ref{def:randomgraphs}. Let $T$ be the constant from Proposition~\ref{prop:fewpaths}. Remove
all $T$-bad pairs from $G$ to obtain a subgraph $G'$ of $G$. Note that for each pair in $\mathcal{B}_T$ which is removed from $G$, at most $2n$ edges are removed. 

Since $f_1,\ldots, f_{\ell-1}$ are chosen independently, \eqref{eq:edgeprobability} implies that the expected number of edges in $G$ is 
\[
q^{\ell}\cdot q^{\ell} \cdot \left(\frac{1}{q}\right)^{\ell-1} = q^{\ell+1}.
\]
Therefore by Proposition~\ref{prop:fewpaths}, we have 
\[
\E\bigl[e(G')\bigr] \geq q^{\ell+1} - 2n\E\bigl[\abs{\mathcal{B}_T}\bigr] = q^{\ell+1} - O(n).
\]
Since $2q^\ell \sim n$, for $n$ large enough we have $\E\bigl[e(G')\bigr] \geq \frac{1}{4}n^{1+1/\ell}$, and so a graph with the desired properties exists.
\end{proof}

We now construct our $\Theta_{\ell, t}$-free graphs. Given a graph $G'$, an \emph{$m$-blowup of $G'$} is obtained by replacing every vertex of $G'$ with an independent set of size $m$ and replacing each edge of $G'$ with a copy of $K_{m,m}$. Note that an $m$-blowup of $G'$ has $m^2e(G')$ edges. If $G$ is a blowup of $G'$, for $u\in V(G)$ and $v\in V(G')$, we say that $v$ is a {\em supervertex} of $u$ if $u$ is in the independent set which replaced $v$.

\begin{proof}[Proof of Theorem~\ref{thm:oddconstruction}]
Let $\ell\geq 3$ be odd, and let $T$ be as above. With foresight, set $m\eqdef\lfloor\frac{t-1}{T}\rfloor$. Let $G'$ be the graph on $\frac{n}{m}$ vertices whose existence is guaranteed by Theorem~\ref{thm:fewpaths}. So $G'$ has at least $\frac{1}{4}\left(\frac{n}{m}\right)^{1+1/\ell}$ edges and no $T$-bad pair. Let $G$ be an $m$-blowup of $G'$. To show that $G$ is $\Theta_{\ell, t}$-free, let $x$ and $y$ be vertices in $G$ and let 
\begin{align*}
P_1=xu_1^{1}&\cdots u_{\ell-1}^1y\\
\vdots &\\
P_R=xu_1^R&\cdots u_{\ell-1}^Ry
\end{align*}
be $R$ internally disjoint paths of length~$\ell$ from $x$ to~$y$. Since $\ell$ is odd and $G'$ is bipartite, $x$ and $y$ have distinct supervertices in $G'$, call them $x'$ and $y'$. For $1\leq i\leq \ell-1$ and $1\leq j\leq R$, let $v_i^j \in V(G')$ be the supervertex of $u_i^j \in V(G)$. Now consider the multiset
\begin{align*}
P_1'=x'v_1^{1}&\cdots v_{\ell-1}^1y'\\
\vdots &\\
P_R'=x'v_1^R&\cdots v_{\ell-1}^Ry'.
\end{align*}
This is a multiset of $R$ not necessarily disjoint or distinct walks of length $\ell$ from $x'$ to $y'$ in $G'$. Removing cycles from these walks, we obtain a multiset of $R$ paths of length at most $\ell$ between $x'$ and $y'$ in $G'$. Although these paths are not necessarily disjoint or distinct, since $G$ is an $m$-blowup of $G'$ and since $P_1,\dotsc,P_R$ are internally disjoint, each vertex besides $x'$ and $y'$ may appear in the multiset of $G'$-paths at most $m$ times. In particular, each distinct $G'$-path may appear at most $m$ times. Since $G'$ has at most $T$ paths of length at most $\ell$ between $x'$ and $y'$, we have that 
\[
  R \leq T m<t
\]
by the choice of $m$.

So, $G$ is a graph on $n$ vertices with no $\Theta_{\ell, t}$ and at least 
\[
\frac{1}{4}\left(\frac{n}{m}\right)^{1+1/\ell}m^2 = \frac{1}{4}n^{1+1/\ell} m^{1-1/\ell} = \Omega_\ell \left(t^{1-1/\ell}n^{1+1/\ell}\right)
\]
edges.
\end{proof}

\section{Lower bound for even length paths}\label{sec:evenconstruction}
Let $h$ be a parameter to be chosen later. Let $G_1,\dotsc,G_h$ be $h$ independent random algebraic graphs with parts $U=V=\F_q^{\ell}$, chosen as in Definition~\ref{def:randomgraphs}.
Consider the multigraph $\overline{G}$ which is the union of all the~$G_i$'s.  Call a pair of vertices \emph{$T$-bad} if they are joined by at least
$T$ paths of length at most $\ell$ in $\overline{G}$.
By Proposition \ref{prop:fewpaths2} (proved in Section~\ref{sec:boring}) 
there are constants $T=T(\ell)$ and $C=C(\ell)$ such that the expected number of $Th^{\ell}$-bad  pairs is at most $Ch^{\ell}$.
Let $G$ be obtained from $\overline{G}$ by removing the multiple edges.

The expected number of edges in the multigraph $\overline{G}$ is $h\cdot q^{\ell+1} = h\left(\frac{n}{2}\right)^{1+1/\ell}$. Let $M$ be the number of multiple edges. Then 
\[
\E[M] \leq n^2 h^2 \left(q^{-\ell-1}\right)^2 = o(n).
\]
Remove from $G$ all $T h^{\ell}$-bad pairs of vertices. Doing this removes at most $2n$ edges per pair removed. 
The expected number of edges in the obtained graph is at least
\[
h\left(\frac{n}{2}\right)^{1+1/\ell} - 2C h^{\ell} n - o(n).
\]

Choosing $h=\left(\frac{t}{T}\right)^{1/\ell}$ shows that there is a $\Theta_{\ell, t}$-free graph with $\Omega_\ell\left(t^{1/\ell}n^{1+1/\ell}\right)$ edges and at most $n$ vertices.

\section{Analysis of the random algebraic construction}\label{sec:boring}
Here we prove the bound, whose proof we deferred, on the number of $T$-bad pairs. Recall that
$G_1,\dotsc,G_h$ are independent random algebraic graphs with parts $U = V = \F_q^\ell$, and $\overline{G}$ 
is the multigraph which is the union of the~$G_i$'s. As before, a pair of vertices is \emph{$T$-bad} if it is joined by
at least $T$ paths of length at most~$\ell$.

Let $\mathcal{B}_T$ be the set of all $T$-bad pairs in $\overline{G}$.
\begin{proposition}\label{prop:fewpaths2}
There exist constants $T=T(\ell)$ and $C=C(\ell)$ such that 
\[
\E\bigl[\abs{\mathcal{B}_{T h^{\ell}}}\bigr] \leq C h^{\ell}.
\]
\end{proposition}

\begin{proof}[Proof of Proposition~\ref{prop:fewpaths2}]
Let $r\leq \ell$ and  $(i_1,\ldots, i_r) \in [h]^r$ be fixed. A path made of edges $e_1,\dotsc,e_r$ (in order) is \emph{of type $(i_1,\ldots, i_r)$}
if $e_j\in E(G_{i_j})$. For a type $I$, a pair of vertices $x,y$ is \emph{$(T,I)$-bad} if there are $T$ paths of type $I$ 
between $x$ and~$y$. We will show that there is a constant $T=T(\ell)$ such that, for each fixed type~$I$, the expected number of $(T/\ell,I)$-bad pairs is $O_\ell(1)$.
Since the total number of types is $\sum_{r\leq \ell} h^r\leq \ell h^r$, the proposition will follow by the linearity of expectation.

We will need the fact that if the degrees of random polynomials are large enough, then 
the values of these polynomials in a small set are independent. Specifically, because of the way we defined graphs $G_i$, 
we are interested in the probabilities that the polynomials vanish on a given set.

\begin{lemma}\label{lem:independence}[Lemma 2.3 in \cite{BukhConlon} and Lemma 2 in \cite{Conlon}]
Suppose that $q\geq \binom{m}{2}$ and $d\geq m-1$. Then if $f$ is a random polynomial from $\mathcal{P}_d^t$ and $x_1,\ldots, x_m$ are fixed distinct points in $\F_q^t$, then 
\[
\mathbb{P}\left[f(x_1) = \cdots = f(x_m) = 0\right] = \frac{1}{q^m}.
\]
\end{lemma}

We need to estimate the expected number of short paths between pairs of vertices. To this end, let $x$ and $y$ be fixed vertices in $G$, let $I=(i_1,\ldots, i_r)$ be fixed, and let $S_r$ be the set of paths of type $I$ between $x$ and~$y$. We use an argument of Conlon \cite{Conlon}, to estimate the $2\ell$'th moment of~$S_I$. 

The $|S_I|^{2\ell}$ counts ordered collections of $2\ell$ paths of type $I$ from $x$ to~$y$.
Let $P_{m,r}$ be all such ordered collections of paths in $K_{q^{\ell},q^{\ell}}$ whose union has exactly $m$ edges. 
Note that $m\leq 2\ell\cdot r\leq 2\ell^2\leq d$. Conlon showed \cite[p.5]{Conlon} that every collection in $P_{m,r}$ spans at least $(r-1)m/r$
vertices other than $x$ and $y$.

By Lemma~\ref{lem:independence} and independence between different $G_i$'s, the probability that a given collection in $P_{m,r}$ is contained in $\overline{G}$
is~$q^{-(\ell-1)m}$. From Conlon's bound on the number of internal vertices it follows that
\[
  \abs{P_{m,j}} \leq q^{\ell m(j-1)/j}.
\]
Therefore,
\[
\E\left[\abs{S_I}^{2\ell}\right] = \sum_{m=1}^{2\ell^2} \abs{P_{m,r}}q^{-(\ell-1)m} \leq \sum_{m=1}^{2\ell^2} q^{\ell m - \frac{\ell m}{r}} q^{-\ell m + m} \leq \sum_{m=1}^{2\ell^2} 1,
\]
where the last inequality uses $r\leq \ell$.

We next show that $|S_I|$ is either bounded or is of order at least~$q$. To do this, we must describe the paths as a points on appropriate varieties. We write $\overline{\F}_q$ for the algebraic closure of $\F_q$. A variety over $\overline{\F}_q$ is a set 
\[
W = \{x\in \overline{\F}_q^t: f_1(x) = \cdots = f_s(x) = 0\}
\]
where $f_1,\ldots,f_s\colon \overline{\F}_q^t \to \overline{\F}_q$ are polynomials. We say that $W$ is \emph{defined over $\F_q$} if the coefficients of the polynomials are in $\F_q$ and we let $W(\F_q) = W\cap \F_q$. We say $W$ has \emph{complexity at most $M$} if $s$, $t$, and the degree of each polynomial are at most~$M$. We need the following lemma of Bukh and Conlon \cite{BukhConlon}.
\begin{lemma}\label{lem:smallorlarge}[Lemma 2.7 in \cite{BukhConlon}]
Suppose $W$ and $D$  are varieties over $\overline{\F}_q$ of complexity at most $M$ which are defined over $\F_q$. Then one of the following holds:
\begin{itemize}
\item $|W(\F_q) \setminus D(\F_q)| \leq c_M$, where $c_M$ depends only on $M$, or
\item $|W(\F_q) \setminus D(\F_q)| \geq q\left(1 - O_M(q^{-1/2})\right)$.
\end{itemize}
\end{lemma}

Note that $S_I$ is a subset of a variety. Indeed, suppose $x\in U$ and $y\in V$ (if $r$ is odd) or $y\in U$ (if $r$ is even) be the two endpoints. 
Let 
\[
  W \eqdef \{ (u_0,\dotsc, u_r)\in (\F_q^\ell)^{r+1} : u_0=x,\ u_r=y,\  f_k^{i_1}(u_0, u_1) = \cdots = f_k^{i_r}(u_{r-1}, u_r)=0,\ 1\leq k\leq \ell-1\},
\]
where $f_k^i$ is the $k$'th random polynomial used to define the random graph~$G_i$.

The set $W(\F_q)$ is nothing but the set of \emph{walks} of type $I$ from $x$ to~$y$. To obtain $S_I$ we need to exclude the walks that are not paths.
To that end, define
\[
  D_{a,b} \eqdef W\cap \{(u_0,\dotsc, u_r): u_a = u_b\}\qquad\text{ for }0\leq a < b \leq r,
\]
and set
$  D\eqdef \bigcup_{a,b} D_{a,b}$,
which is a variety since the union of varieties is a variety. Furthermore, its complexity is bounded
since it is defined by polynomials that are products of polynomials defining $D_{a,b}$'s.

We then have that 
\[
S_I = W(\F_q) \setminus D(\F_q).
\]
Since complexity of both $W$ and $D$ is bounded, Lemma~\ref{lem:smallorlarge} implies that either $|W(\F_q) \setminus D(\F_q)| \leq c_j$ or $|W(\F_q) \setminus D(\F_q)| \geq q\left(1-O_r(q^{-1/2})\right)$ where $c_r$ is a constant depending only on~$r$. In particular, there is a constant $T_r$ such that, for $q$ large enough, we have either $\abs{S_I} \leq T_r$ or $\abs{S_I} \geq \frac{q}{2}$. Since $\E\bigl[\abs{S_I}^{2\ell}\bigr] \leq 2\ell^2$, Markov's inequality gives that 
\begin{equation}\label{eq:tail}
\mathbb{P}\bigl[|S_I| > T_r \bigr] = \mathbb{P}\left[|S_I| \geq \frac{q}{2}\right] = \mathbb{P}\left[|S_I|^{2\ell} \geq (q/2)^{2\ell}\right] \leq \frac{\E(|S_I|^{2\ell})}{(q/2)^{2\ell}} = O_r\left(q^{-2\ell}\right).
\end{equation}

Upon letting $T\eqdef \ell\cdot \max_{r\leq \ell} T_r$, inequality \eqref{eq:tail} implies that the expected number of $(T/\ell,I)$-bad pairs is at most $O_r\left(\abs{V}\abs{U}q^{-2\ell}\right)=O_{r}(1)$.
\end{proof}

\bibliographystyle{plain}
\bibliography{bib}

\end{document}